\documentclass[12pt]{article}
\usepackage{amsmath,amsfonts,amssymb,amsthm,txfonts,hyperref,graphicx,url}
\usepackage[a4paper,margin=2.5truecm]{geometry}
\newtheorem{theorem}{Theorem}
\newtheorem{lemma}{Lemma}
\newtheorem{question}{Question}
\newcommand{\N}{\mathbb N}
\newcommand{\Z}{\mathbb Z}
\title{On systems of Diophantine equations with a large number of integer solutions}
\date{}
\author{Apoloniusz Tyszka}
\begin{document}
\begin{sloppypar}
\maketitle
\begin{abstract}
Let \mbox{$E_n=\{x_i+x_j=x_k,~x_i \cdot x_j=x_k \colon~i,j,k \in \{1,\ldots,n\}\}$}.
For each integer \mbox{$n \geqslant 13$}, \mbox{J. Browkin} defined a system \mbox{$B_n \subseteq E_n$}
which has exactly $b_n$ solutions in integers \mbox{$x_1,\ldots,x_n$}, where
\mbox{$b_n \in \N \setminus \{0\}$} and the sequence $\left\{b_n\right\}_{n=13}^\infty$ rapidly tends to infinity.
For each integer \mbox{$n \geqslant 12$}, we define a system \mbox{$T_n \subseteq E_n$} which has exactly
$t_n$ solutions in integers \mbox{$x_1,\ldots,x_n$}, where \mbox{$t_n \in \N \setminus \{0\}$} and
\mbox{$\displaystyle \lim\limits_{n \to \infty} \frac{t_n}{b_n}=\infty$}.
\end{abstract}
\vskip 1.2truecm
\par
For a \mbox{non-negative} integer $n$, let $r_k(n)$ denote the number of representations of $n$ as a sum of $k$ squares of integers.
Let \mbox{$E_n=\{x_i+x_j=x_k,~x_i \cdot x_j=x_k \colon~i,j,k \in \{1,\ldots,n\}\}$}. For an integer \mbox{$n \geqslant 13$},
let \mbox{$B_n$} denote the following system of equations (\cite{Browkin}):
\[
\left\{\begin{array}{rcl}
\forall i \in \{1,\ldots,n-13\} ~x_i \cdot x_i &=& x_{i+1} \\
x_{n-11}+x_{n-11} &=& x_{1} \\
x_{n-11} \cdot x_{n-11} &=& x_{1} \\
x_{n-10} \cdot x_{n-10} &=& x_{n-9} \\
x_{n-8} \cdot x_{n-8} &=& x_{n-7} \\
x_{n-6} \cdot x_{n-6} &=& x_{n-5} \\
x_{n-4} \cdot x_{n-4} &=& x_{n-3} \\
x_{n-9}+x_{n-7} &=& x_{n-2} \\
x_{n-5}+x_{n-3} &=& x_{n-1} \\
x_{n-2}+x_{n-1} &=& x_{n} \\
x_{n-11}+x_{n-12} &=& x_{n}
\end{array}\right.
\]
The system $B_n$ is contained in $E_n$ and $B_n$ equivalently expresses that
\[
x_1=\ldots=x_n=0
\]
or
\[
\left(x_{n-11}=2\right) \wedge \left(x_n=2+2^{\textstyle 2^{n-12}}=x_{n-10}^2+x_{n-8}^2+x_{n-6}^2+x_{n-4}^2\right) \wedge
\]
\[
all~the~other~variables~are~uniquely~determined~by~the~above~conjunction~and~the~equations~of~B_n{\rm .}
\]
\begin{theorem}\label{the1}(\cite{Browkin})
The system $B_n$ has exactly $1+r_4\left(2+2^{\textstyle 2^{n-12}}\right)=1+8\sigma\left(2+2^{\textstyle 2^{n-12}}\right)$
solutions in integers \mbox{$x_1,\ldots,x_n$}, where $\sigma(\cdot)$ denote the sum of positive divisors.
\end{theorem}
\par
For a positive integer \mbox{$n \geqslant 12$}, let $T_n$ denote the following system of equations:
\[
\left\{\begin{array}{rcl}
\forall i \in \{1,\ldots,n-12\} ~x_i \cdot x_i &=& x_{i+1} \\
x_{n-10} \cdot x_{n-10} &=& x_{n-10} \\
x_{n-10}+x_{n-10} &=& x_{n-9} \\
x_{n-8}+x_{n-9} &=& x_{1} \\
x_{n-8} \cdot x_{n-7} &=& x_{n-11} \\
x_{n-10} \cdot x_{n-7} &=& x_{n-7} \\
x_{n-6} \cdot x_{n-6} &=& x_{n-5} \\
x_{n-4} \cdot x_{n-4} &=& x_{n-3} \\
x_{n-2} \cdot x_{n-2} &=& x_{n-1} \\
x_{n-5}+x_{n-3} &=& x_n \\
x_{n-1}+x_n &=& x_{n-11} 
\end{array}\right.
\]
\begin{theorem}\label{the2}
The system $T_n$ is contained in $E_n$ and $T_n$ has exactly
\mbox{$1+\sum\limits_{\textstyle k=0}^{\textstyle 2^{n-12}} r_3\left(\left(2 \pm 2^k\right)^{\textstyle 2^{n-12}}\right)$}
solutions in integers \mbox{$x_1,\ldots,x_n$}.
\end{theorem}
\begin{proof}
The system $T_n$ equivalently expresses that 
\[
x_1=\ldots=x_n=0
\]
or
\[
\left(x_{n-10}=1\right) \wedge
\left(\left(x_1-2\right) \cdot x_{n-7}=x_{n-11}=x_1^{\textstyle 2^{n-12}}=x_{n-6}^2+x_{n-4}^2+x_{n-2}^2\right) \wedge
\]
\[
all~the~other~variables~are~uniquely~determined~by~the~above~conjunction~and~the~equations~of~T_n{\rm .}
\]
In the second case, by the polynomial identity
\[
x_1^{\textstyle 2^{n-12}}=2^{\textstyle 2^{n-12}}+\left(x_1-2\right) \cdot
\sum\limits_{\textstyle k=0}^{\textstyle 2^{n-12}-1} 2^{\textstyle 2^{n-12}-1-k} \cdot x_1^k
\]
we obtain that 
\[
2^{\textstyle 2^{n-12}}=\left(x_1-2\right) \cdot \left(x_{n-7}-
\sum\limits_{\textstyle k=0}^{\textstyle 2^{n-12}-1} 2^{\textstyle 2^{n-12}-1-k} \cdot x_1^{k}\right)
\]
Hence, \mbox{$x_1-2$} divides \mbox{$2^{\textstyle 2^{n-12}}$}.
Therefore, \mbox{$x_1 \in \left\{2 \pm 2^{k} \colon k \in \left[0,2^{n-12} \right] \cap \Z \right\}$}.
Consequently,
\[
x_{n-11}=x_1^{\textstyle 2^{n-12}} \in
\left\{\left(2 \pm 2^{k}\right)^{\textstyle 2^{n-12}} \colon k \in \left[0,2^{n-12} \right] \cap \Z \right\}
\]
Since the last five equations of $T_n$ equivalently express that \mbox{$x_{n-11}=x_{n-6}^2+x_{n-4}^2+x_{n-2}^2$},
the proof is complete.
\end{proof}
\vskip 0.01truecm
\par
The following lemma is a consequence of Siegel's theorem (\cite{Siegel}).
\begin{lemma}\label{lem1}
(\cite[p.~119]{Freeden},~\cite[p.~271]{Michel}) For every \mbox{$\varepsilon \in (0,\infty)$}
there exists $c\left(\varepsilon\right) \in (0,\infty)$ such that
$r_3 \left(4^{\textstyle s} \cdot m \right) \geqslant c\left(\varepsilon\right) \cdot m^{\textstyle \frac{1}{2}-\varepsilon}$
for every non-negative integer $s$ and every positive integer
$m \not\in \left\{4k \colon k \in \Z \right\} \cup \left\{8k+7 \colon k \in \Z \right\}$.
\end{lemma}
\par
Let $b_n$ denote the number of integer solutions of $B_n$, and let $t_n$ denote the number of integer solutions of $T_n$.
\begin{theorem}\label{the3}
\mbox{$\displaystyle \lim\limits_{n \to \infty} \frac{t_n}{b_n}=\infty$}.
\end{theorem}
\begin{proof}
Let an integer $n$ is greater than $12$. By Theorem~\ref{the2},
\[
t_n>r_3\left(\left(2+2^{\textstyle 2^{n-12}}\right)^{\textstyle 2^{n-12}}\right)=
r_3\left(4^{\textstyle 2^{n-13}} \cdot \left(1+2^{\textstyle 2^{n-12}-1}\right)^{\textstyle 2^{n-12}}\right)
\]
By Lemma~\ref{lem1}, for each \mbox{$\varepsilon \in (0,\infty)$} there
exists \mbox{$c(\varepsilon) \in (0,\infty)$} such that
\[
r_3\left(4^{\textstyle 2^{n-13}} \cdot \left(1+2^{\textstyle 2^{n-12}-1}\right)^{\textstyle 2^{n-12}} \right)
\geqslant c(\varepsilon) \cdot
\left(\left(1+2^{\textstyle 2^{n-12}-1} \right)^{\textstyle 2^{n-12}} \right)^{\textstyle \frac{1}{2}-\varepsilon}
\]
for every integer \mbox{$n>12$}. We take \mbox{$\varepsilon=\displaystyle\frac{1}{4}$}.
Since \mbox{$2^{n-12}-1 \geqslant 2^{n-13}$}, we get
\[
c(\varepsilon) \cdot \left(\left(1+2^{\textstyle 2^{n-12}-1} \right)^
{\textstyle 2^{n-12}} \right)^{\textstyle \frac{1}{2}-\varepsilon} >
c\left(\frac{1}{4}\right) \cdot \left(2^{\textstyle 2^{n-13}}\right)^
{\textstyle 2^{n-12} \cdot \frac{1}{4}}=c\left(\frac{1}{4}\right) \cdot 2^{\textstyle 2^{2n-27}}
\]
Since \mbox{$1<\sigma\left(2+2^{\textstyle 2^{n-12}}\right)$} and
\mbox{$2+2^{\textstyle 2^{n-12}}<2^{\textstyle 2^{n-11}}$}, Theorem~\ref{the1} gives:
\[
b_{n}=1+8\sigma\left(2+2^{\textstyle 2^{n-12}}\right)<9\sigma\left(2+2^{\textstyle 2^{n-12}}\right) \leqslant
9 \cdot \sum\limits_{k=1}^{\textstyle 2+2^{\textstyle 2^{n-12}}} k <
9 \cdot \sum\limits_{k=1}^{\textstyle 2^{\textstyle 2^{n-11}}} k <
9 \cdot 2^{\textstyle 2^{n-11}} \cdot 2^{\textstyle 2^{n-11}}=
\]
\[
9 \cdot 2^{\textstyle 2^{n-10}}<2^{\textstyle 2^{n-9}}
\]
\vskip 0.01truecm
\noindent
Therefore, $\displaystyle \frac{t_n}{b_n}>\frac{c\left(\displaystyle\frac{1}{4}\right) \cdot
2^{\textstyle 2^{2n-27}}}{2^{\textstyle 2^{n-9}}}$. This quotient tends to infinity when $n$
tends to infinity, which completes the proof.
\end{proof}
\newpage
The following {\sl Mathematica} code first computes decimal approximations
of \mbox{$\displaystyle \frac{t_n}{b_n}$} for all integers \mbox{$n \in [13,19]$}.
\vskip 0.01truecm
\begin{center}
\includegraphics[width=\hsize]{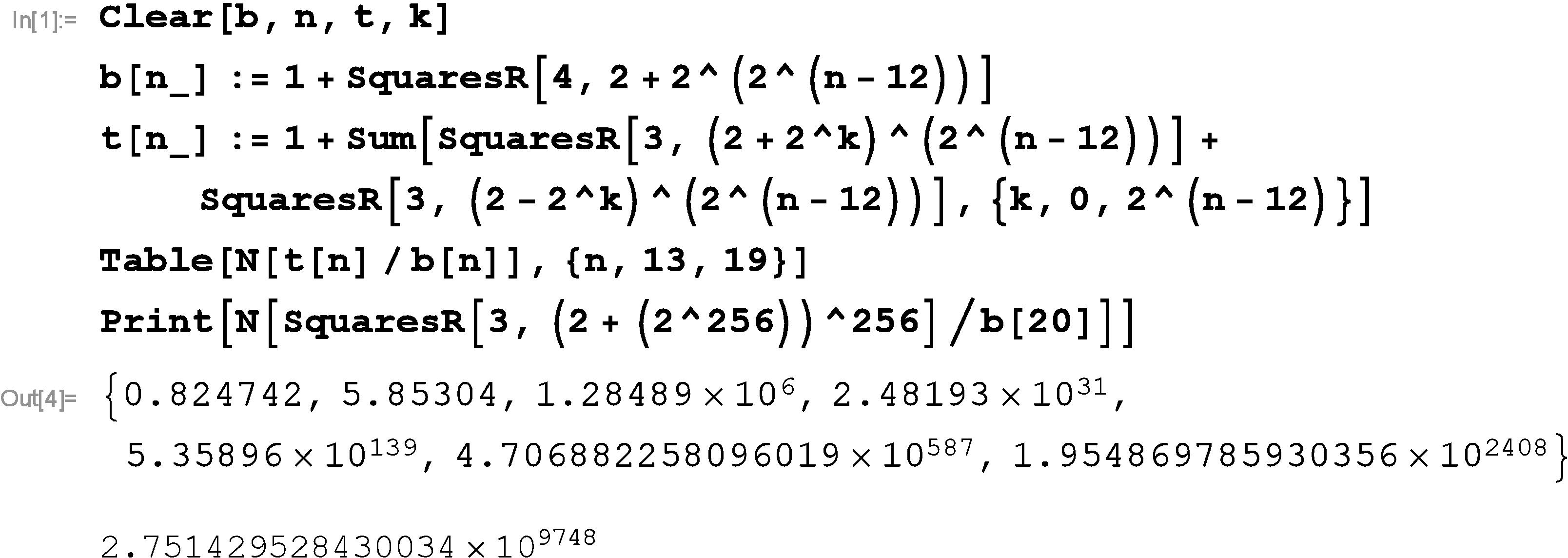}
\end{center}
The output results show that \mbox{$t_n>b_n$} for every integer \mbox{$n \in [14,19]$}.
By Theorem~\ref{the2},
\[
\frac{t_{20}}{b_{20}}>\frac{r_3\left(\left(2+2^{\textstyle 2^{20-12}}\right)^{\textstyle 2^{20-12}}\right)}{b_{20}}=
\frac{r_3\left(\left(2+2^{256}\right)^{256}\right)}{b_{20}}
\]
The last command of the code finds the decimal approximation of the last quotient.
Hence, \mbox{$\displaystyle \frac{t_{20}}{b_{20}}>2.75 \cdot 10^{9748}$}.
It seems that \mbox{$t_n>b_n$} for every integer \mbox{$n \geqslant 21$}, although this remains unproven.
\vskip 0.2truecm
\par
Let us define the height of a rational number $\frac{p}{q}$ by \mbox{${\rm max}(|p|,|q|)$}
provided $\frac{p}{q}$ is written in lowest terms. Let us define the height of a rational
tuple \mbox{$(x_1,\ldots,x_n)$} as the maximum of $n$ and the heights of the numbers \mbox{$x_1,\ldots,x_n$}.
\vskip 0.2truecm
\par
For an integer \mbox{$n \geqslant 4$}, let $S_n$ denote the following system of equations:
\[
\left\{\begin{array}{rcl}
\forall i \in \{1,\ldots,n-4\} ~x_i \cdot x_i &=& x_{i+1} \\
x_{n-2}+1 &=& x_1 \\
x_{n-1}+1 &=& x_{n-2} \\
x_{n-1} \cdot x_n &=& x_{n-3}
\end{array}\right.
\]
\begin{theorem}\label{the4}(\cite{Tyszka})
The system $S_n$ is soluble in positive integers and has only finitely many integer solutions.
Each integer solution \mbox{$(x_1,\ldots,x_n)$} satisfies
\mbox{$|x_1|,\ldots,|x_n| \leqslant \left(2+2^{\textstyle 2^{n-4}}\right)^{\textstyle 2^{n-4}}$}.
The following equalities
\begin{displaymath}
\begin{array}{rcl}
\forall i \in \{1,\ldots,n-3\} ~x_i &=& \left(2+2^{\textstyle 2^{n-4}}\right)^{\textstyle 2^{i-1}} \\
x_{n-2} &=& 1+2^{\textstyle 2^{n-4}} \\
x_{n-1} &=& 2^{\textstyle 2^{n-4}} \\
x_n &=& \left(1+2^{\textstyle 2^{n-4}-1}\right)^{\textstyle 2^{n-4}}
\end{array}
\end{displaymath}
define the unique integer solution whose height is maximal.
\end{theorem}
\newpage
\noindent
{\bf Conjecture.} {\em If an integer $n$ is sufficiently large and a system
\[
U \subseteq \{x_i+1=x_k,~x_i \cdot x_j=x_k:~i,j,k \in \{1,\ldots,n\}\}
\]
has only finitely many solutions in positive integers \mbox{$x_1,\ldots,x_n$},
then each such solution \mbox{$(x_1,\ldots,x_n)$} satisfies
\mbox{$x_1,\ldots,x_n \leqslant \left(2+2^{\textstyle 2^{n-4}}\right)^{\textstyle 2^{n-4}}$}.}
\vskip 0.2truecm
\par
A bit stronger version of the Conjecture appeared in \cite{Tyszka}.
The Conjecture implies that there is an algorithm which takes as input a Diophantine equation,
returns an integer, and this integer is greater than the heights of integer (\mbox{non-negative} integer,
positive integer, rational) solutions, if the solution set is finite (\cite{Tyszka}).
\vskip 0.2truecm
\par
Let us pose the following two questions:
\begin{question}\label{que1}
Is there an algorithm which takes as input a Diophantine equation, returns an integer, and this
integer is greater than the heights of integer solutions, if the solution set is finite?
\end{question}
\begin{question}\label{que2}
Is there an algorithm which takes as input a Diophantine equation, returns an integer, and this
integer is greater than the number of integer solutions, if the solution set is finite?
\end{question}
\par
Obviously, a positive answer to Question~\ref{que1} implies a positive answer to Question~\ref{que2}.
\begin{lemma}\label{lem2}
Let $d$ denote the maximal height of an integer solution of a Diophantine equation \mbox{$D(x_1,\ldots,x_n)=0$}
whose solution set in integers is \mbox{non-empty} and finite. We claim that the number of integer solutions
to the equation
\[
D^2(x_1,\ldots,x_n)+\left(n^2+x_1^2+\ldots+x_n^2-u_1^2-u_2^2-u_3^2-u_4^2-v_1^2-v_2^2-v_3^2-v_4^2\right)^2=0
\]
is finite and greater than $d$.
\end{lemma}
\begin{proof}
There exists an integer tuple \mbox{$(a_1,\ldots,a_n)$} such that \mbox{$D(a_1,\ldots,a_n)=0$}
and \mbox{${\rm max}\left(n,|a_1|,\ldots,|a_n|\right)=d$}. The equation \mbox{$n^2+a_1^2+\ldots+a_n^2=x+y$}
has \mbox{$n^2+a_1^2+\ldots+a_n^2+1$} solutions in \mbox{non-negative} integers $x$ and $y$. Since
\mbox{$n^2+a_1^2+\ldots+a_n^2+1 \geqslant d^2+1>d$}, the claim follows from Lagrange's \mbox{four-square} theorem.
\end{proof}
\begin{theorem}\label{the5}
A positive answer to Question~\ref{que2} implies a positive answer to Question~\ref{que1}.
\end{theorem}
\begin{proof}
In order to compute an upper bound on the heights of integer solutions to a Diophantine equation
\mbox{$D(x_1,\ldots,x_n)=0$} with a finite number of integer solutions, we compute an upper bound
on the number of integer solutions to the equation
\[
D^2(x_1,\ldots,x_n)+\left(n^2+x_1^2+\ldots+x_n^2-u_1^2-u_2^2-u_3^2-u_4^2-v_1^2-v_2^2-v_3^2-v_4^2\right)^2=0
\]
By Lemma~\ref{lem2}, this number is greater than the heights of integer solutions to \mbox{$D(x_1,\ldots,x_n)=0$}.
\end{proof}

\vskip 0.01truecm
\noindent
Apoloniusz Tyszka\\
University of Agriculture\\
Faculty of Production and Power Engineering\\
Balicka 116B, 30-149 Krak\'ow, Poland\\
E-mail address: \url{rttyszka@cyf-kr.edu.pl}
\end{sloppypar}
\end{document}